\documentclass{amsart}

\usepackage{graphicx} \usepackage{tikz-cd} 

\usepackage{amsmath, amssymb, amsthm, amsfonts}  
\usepackage[pdftex,  colorlinks=true]{hyperref} 

\usepackage{xcolor} 

\usepackage{yhmath}

\usepackage{accents}

\usepackage{caption}
\usepackage{subcaption}
 
\newtheorem{theorem}{Theorem}[section]

\newtheorem{prop}[theorem]{Proposition}
\newtheorem{proposition}[theorem]{Proposition}
\newtheorem{corollary}[theorem]{Corollary}

\theoremstyle{definition}
\newtheorem{remark}[theorem]{Remark}
\newtheorem{example}[theorem]{Example}
\newtheorem{question}[theorem]{Question}

\theoremstyle{definition}
\newtheorem{definition}[theorem]{Definition}

\DeclareMathOperator{\Cay}{Cay}

\DeclareMathOperator{\dist}{\mathsf{dist}}

\DeclareMathOperator{\wCop}{wCop}

\makeatletter
\newcommand*{\bigcdot}{}
\DeclareRobustCommand*{\bigcdot}{%
  \mathbin{\mathpalette\bigcdot@{}}%
}
\newcommand*{\bigcdot@scalefactor}{.5}
\newcommand*{\bigcdot@widthfactor}{1.15}
\newcommand*{\bigcdot@}[2]{%
  \sbox0{$#1\vcenter{}$}
  \sbox2{$#1\cdot\m@th$}%
  \hbox to \bigcdot@widthfactor\wd2{%
    \hfil
    \raise\ht0\hbox{%
      \scalebox{\bigcdot@scalefactor}{%
        \lower\ht0\hbox{$#1\bullet\m@th$}%
      }%
    }%
    \hfil
  }%
}
\makeatother

\newcounter{ccomments}

\newcounter{ecomments}

\begin{document}

\title[Weak cops in Wreath Products]{The Lamplighter groups have infinite weak cop number}

\author[ ]{Anders Cornect} 
\address{Waterloo, ON, Canada N2L 3G1}
\email{acornect@uwaterloo.ca}

\author[]{Eduardo Mart\'inez-Pedroza }  
\address{Memorial University of Newfoundland, St. John's, NL, Canada}
\email{emartinezped@mun.ca}

\date{\today}

\begin{abstract}
The weak cop number of a graph,  a variation of the cop number,  is an invariant suitable for infinite graphs and is a quasi-isometric invariant. While for any $m\in\mathbb{Z}_+\cup\{\infty\}$ there exist locally finite infinite graphs with weak cop number $m$, it is an open question whether there exists locally finite vertex transitive graphs whose weak cop number is different than $1$ and $\infty$. We test this question on Cayley graphs of wreath products; these are objects known for their exotic geometries. We prove that Cayley graphs of wreath products of nontrivial groups by infinite groups have infinite weak cop number.  The result is proved by defining a new pursuit and evasion game  and proving the existence of strategies for the evader. We also include a short argument that Cayley graphs of Thompson's group $F$ have infinite weak cop number.      
\end{abstract}

\maketitle


 \section{Introduction}

There has been recent interest on quasi-isometric invariants of graphs defined via combinatorial games, and their connections to geometric group theory, see for example~\cite{ABK20, ABGK23, Le19, Lee2023, MP23}.

The Cops and Robber game was introduced independently in the late 1970's and early 1980's by different researchers; among these were the works of  Quilliot~\cite{Quilliot1978} and Nowakowski and Winkler~\cite{Nowakowski1983}. This is a perfect information two player game on an undirected graph, where one player controls a set of cops and the other one controls a single robber. On the graph each cop and the robber choose a vertex to occupy, with the cops choosing first. The game then alternates between cops and the robber moving along adjacent vertices, with the cops moving first. The cops win if, after a finite number of rounds, a cop occupies the same vertex as the robber;  the robber wins if he can avoid capture indefinitely. The cop number of a graph is the minimum number of cops necessary to always capture a robber.  

Lee et al.~\cite{Lee2023} introduced a variation of the cops and robber game called  \emph{Weak cops and Robbers}. This is a two-player game where one of the players controls a finite set of cops, while the other controls a single robber. The objective of the cops is to protect arbitrarily large finite subgraphs of the underlying graph, subject to some parameters. These parameters are chosen by the players. First, the cops player chooses the number of cops, as well as the cops' speed and reach. The robber player, knowing this information, chooses his speed and challenges the cops to protect a large ball in the graph; the cops choose their initial positions, and then the robber choose his initial position. After these choices have been made, the game starts and the cops and robber move in alternating turns, up to a distance determined by their respective speeds. The cops win the game if at some stage the robber is captured (the robber is within reach of a cop) or from some turn on, they keep the robber outside the ball. 
See Section~ \ref{sec:Background} for a precise definition. 

The \emph{weak cop number} of a connected graph $\Gamma$, denoted $\wCop(\Gamma)$, is the minimum number of cops required for the cop player to always have a winning strategy in the Weak cops and Robbers game on $\Gamma$. If no such finite number exists, then we say that $\wCop(\Gamma) = \infty$. This is an interesting numerical  invariant of infinite graphs, it is preserved by quasi-isometries, for every $m \in \mathbb{N} \cup \{\infty\}$ there exists a graph with weak cop number $m$, and the invariant behaves well with respect to different operations including  products, see the work of  
Lee et al~\cite{Lee2023}. 
They raised the following question given that their examples of graphs with arbitrary weak cop numbers are locally finite but lack symmetry.

 \begin{question}[{\cite[Question L]{Lee2023}}]\label{question:main0}
Is there a connected vertex transitive locally finite graph $\Gamma$ such that $1<\wCop(\Gamma)<\infty$?
\end{question}

A natural source to explore this question is on Cayley graphs of finitely generated groups. In fact, the notion of weak cop number yields an invariant of finitely generated groups, via playing the game on Cayley Graphs with respect to finite generating sets, since any two Cayley graphs with respect to finite generating sets of a common group have the same weak cop number~\cite[Corollary G]{Lee2023}.

\begin{definition}[Weak cop number of a group]
    Let $G$ be a finitely generated group. The weak cop number $\wCop(G)$ is defined as the weak cop number of any Cayley graph of $G$ with respect to a finite generating set. 
\end{definition}

In~\cite{Lee2023}, Lee et al showed that for finitely generated groups,  free groups have weak cop number 1, and non-cyclic free abelian groups have infinite weak cop number. A non-trivial result in their article   is  that one-ended non-amenable groups have infinite weak cop number~\cite[Theorem H]{Lee2023}.  They also stated they following restricted version of Question~\ref{question:main0}. 

  \begin{question}[{\cite[Question K]{Lee2023}}]\label{question:main}
        Does there exist a finitely generated group $G$ with $1 < \wCop(G) < \infty$?
    \end{question}

Question~\ref{question:main} is intriguing. A positive answer would place the weak cop number as a rich  integral invariant of finitely generated groups worth of a deeper study. Even in the negative case, a description of the class of finitely generated groups with weak cop number one does not seem trivial, but it is plausible to expect a positive answer to the following question.

\begin{question}
For finitely generated groups, does having weak cop number one implies being virtually free?
\end{question}

By the results in~\cite{Lee2023}, efforts to understand Question~\ref{question:main} should focus on one-ended amenable groups. In this article, we test this question on some finitely generated groups known for their exotic geometries. Our main result is the following.

\begin{theorem}\label{thm:main}
  The restricted wreath product $G \wr H$ of finitely generated groups  has infinite weak cop number if $G$ is non-trivial and $H$ is infinite.  
\end{theorem}

Some well-known groups covered by Theorem~\ref{thm:main} are the \emph{Lamplighter groups} $\mathbb{Z}_m\wr\mathbb{Z}$ which are known to be amenable and of exponential growth. Let us remark that a result of Baumslag~\cite{Ba60} implies that the  wreath products $G\wr H$ with $G$ non-trivial and $H$ infinite are not finitely presented.

Theorem~\ref{thm:main} is proved by defining a new game called \emph{Lamplighter}. Briefly, this is a perfect information game with at least $P\geq2$ players on a single underlying object called a \emph{streetmap}, essentially a Cayley graph $\Cay(H)$ where the vertices represent lamps that can take on a variety of states given by $\Cay(G)$. One player, the \emph{lamplighter}, moves along a copy of the graph $\Cay(H)$ restricted to a chosen finite subgraph called the \emph{area of play}, changing the states of lamps. The other $P-1$ players, the \emph{copiers}, move along their own copies of $\Cay(H)$, working together in an attempt for one of them to approximate the lamplighter's pattern of lamps. The lamplighter and the group of copiers move alternately, changing the state of some of the lamps close to their current position. The copiers win the game if eventually one of them can approximate the lamplighter's pattern of lamps. Under the assumptions, we show that the lamplighter has a wining strategy for any number of copiers. We then prove that this implies that $\Cay(G\wr H)$ has infinite weak cop number. This new game can be regarded as a variation of the weak cops and robbers where the robber player imposes on himself a finite area of play, instead of being able to freely move on the underlying infinite graph.
A precise definition of the Lamplighter game and the proof of Theorem~\ref{thm:main} are the contents of Section~\ref{sec:Lamp}.

Let us state a question related to Theorem~\ref{thm:main}. The Diestel-Leader graphs $DL(m,n)$ for $m,n\in\mathbb{Z}_+$ are connected, locally finite, vertex-transitive graphs which are regarded as generalizations of Cayley graphs of Lamplighter groups $\mathbb{Z}_n\wr \mathbb{Z}$. It is a remarkable result of Eskin, Fisher and Whyte that $D(m,n)$ is quasi-isometric to a Cayley graph of a finitely generated group if and only if $m=n$; see~\cite{EFW12}. The graph $D(m,m)$ is quasi-isometric to a Cayley graph of the Lamplighter group $\mathbb{Z}_m\wr \mathbb{Z}$ which has infinite weak cop number by Theorem~\ref{thm:main}. 

\begin{question}
 Does $D(m,n)$ have infinite weak cop number for any $m,n\in \mathbb{Z}_+$ ?   
\end{question}

Another finitely generated group known for its exotic geometry is Thompson's group $F$, for background see~\cite{BG84,CFP96}. It  is an outstanding question in group theory whether $F$ is amenable. In this note, we also include short computation of the weak cop number of Thompson's group $F$ communicated to the authors by Francesco Fournier-Facio. 

\begin{theorem}\label{thm:main2}
  Thompson's group $F$ has infinite weak cop number.
\end{theorem}

The proof that $F$ has infinite weak cop number relies on the existence of a group  retraction  $F\to \mathbb{Z}^2$. Then, results in~\cite{Lee2023}, imply that $\wCop(F)\geq\wCop(\mathbb{Z}^2)=\infty$. 

There are other groups with exotic geometries where one can test Question~\ref{question:main}. The authors are not ready to conjecture an answer for this question.

\subsection*{Organization} The rest of the article is organized into three sections. Section~\ref{sec:Background} contains some preliminaries and the precise definition of the weak cop number for graphs. Then  Section~\ref{sec:Lamp} contains the proof of Theorem~\ref{thm:main}.
The last section contains the proof  that Thompson's group $F$ has infinite weak cop number.

\subsection*{Acknowledgements} 
The authors thank Francesco Fournier-Facio,  Florian Lehner and Danny Dyer for comments on preliminary versions of this work. Special thanks to the anonymous referee for suggestions and corrections.  
The first author acknowledges funding by the Natural Sciences and Engineering Research Council of Canada NSERC, via the Undergraduate Student Research Award (USRA). The second author acknowledges funding by the Natural Sciences and Engineering Research Council of Canada NSERC.

\section{Preliminaries} \label{sec:Background}

\subsection{Graph theory language}
   A \emph{graph} $\Gamma$ is a pair $(V,E)$, where $V$ is called the set of \emph{vertices} of $\Gamma$, and $ E\subset \binom{V}{2} $  contains subsets of cardinality 2 in $V$, called the \emph{edges} of $\Gamma$. Let   $V(\Gamma)$ and $E(\Gamma)$ denote the vertex set and the edge set of $\Gamma$  respectively. Two vertices $u,v \in V$ are said to be \emph{adjacent} in $\Gamma$ if $\{u,v\} \in E$. Observe that this definition encompasses \emph{simple} graphs, i.e. graphs that have no edges from a vertex to itself (no loops), and each edge appears at most once in $E$ (no multiple edges). A graph is \emph{trivial} if it has only one vertex, and is \emph{infinite} if it has infinite vertex set. By an \emph{isomorphism} from a graph $G$ to a graph $H$, we mean a bijection $\Phi\colon V(G) \rightarrow V(H)$ such that $\{u,v\} \in E(G)$ if and only if $\left\{\Phi(u),\Phi(v)\right\} \in E(H)$.

By a \emph{path} in a graph $\Gamma$ we mean a sequence of vertices $v_0$, $v_1$, $v_2$, \ldots, $v_k$  such that, for each $i = 1,2,\ldots,k$, $\{v_i, v_{i-1}\}$ is an edge in $\Gamma$. The \emph{length} of such path is defined as $k$. 
A graph is \emph{connected} if there if there is a path between any two vertices. In a connected graph, the length of the shortest path between two vertices $u,v$ is  called the \emph{distance} between them and is denoted by  $\dist_\Gamma(u,v)$. A path $v_0,v_1,\ldots ,v_k$ is a  \emph{geodesic} if $\dist_\Gamma(v_0,v_k)=k$. Note that in a connected graph, there is a geodesic between any pair of points. The \emph{diameter} of $\Gamma$ is the supremum of the distances between any pair of vertices; in particular, a graph can have infinite diameter. Note that if the graph has infinite diameter, then there are geodesics of arbitrarily large length. A graph $\Gamma$ is \emph{locally finite} if for any vertex $v$, the set of vertices at distance one from $v$ is finite. 

By the \emph{infinite path $P_\infty$} we mean the graph with vertex $\mathbb{Z}$ and edge set $\{ (n,n+1)\mid\; n\in \mathbb{Z}\}$. As usual, the \emph{$n$-path $P_n$} is the graph with vertex set $n=\{0,1,\ldots , n-1\}$ and edge set $\{(k,k+1)\mid\; 0\leq k<n-1 \}$; the \emph{$n$-cycle $C_n$} is the graph with vertex set $n$ and edge set $E(P_n)\cup\{ \{0,n-1\}\}$.    

Let $G$ be a group with a generating set $S \subset G$ that does not contain the identity.  The \emph{Cayley graph} of $G$ with respect to $S$, denoted $\Cay(G,S)$, is the graph with $V(\Cay(G,S)) = G$, and $E(\Cay(G,S)) = \left\{\{g, g  s\} \mid \; g \in G, s \in S \right\}$. It is a simple exercise to show that $\Cay(G,S)$ is a connected graph. Observe that if $S$ is finite then $\Cay(G,S)$ is a locally finite graph. In particular, if $G$ is an infinite group and  $S$ is a finite generating set, then $\Cay(G,S)$ is a locally finite, infinite, connected graph, and therefore by K\"onig's lemma, it contains geodesics of arbitrarily large length. 
    
\subsection{Definition of the weak cop number}

    Given a connected graph $\Gamma$, Weak Cops and Robbers is played on $\Gamma$ as follows. There are two players, with one playing the robber, and one playing a set of $n$ cops for some $n \in \mathbb{N}$. Before the game begins, the cops choose two positive integers $\sigma$ and $\rho$, called the cops' \emph{speed} and \emph{reach}, respectively. Knowing these values, the robber then chooses positive integers $\psi$ and $R$, along with a vertex $v$ of $\Gamma$; these parameters are called the robber's \emph{speed}, the \emph{radius of the area of play} and the \emph{center of the area of play} respectively.

    Once the parameters have been chosen, each of the cops choose a vertex for their initial positions. Knowing the choices made by each of the cops, the robber chooses a vertex for their own initial position. The cops and robber move in alternating turns, starting with the cops. On the cops' turn, each cop can move to a vertex at distance at most $\sigma$ from its current position. The robber is \emph{captured} during this turn if any of the cops move to within distance $\rho$ of him. On the robber's turn, he can move to a vertex at distance at most $\psi$ from his current position, provided he has a path to that vertex that contains no vertex within distance $\rho$ from any of the cops. The cops win if they can eventually protect the ball of radius $R$ centered at $v$. This means that either the robber is captured or, beginning on some turn, they can permanently prevent the robber from moving within distance $R$ of the vertex $v$.

    We say that the graph $\Gamma$ is CopWin$(n,\sigma,\rho,\psi,R)$ if for any $v \in V(\Gamma)$, $n$ cops with speed $\sigma$ and reach $\rho$ can eventually protect the ball of radius $R$ centered at $v$. $\Gamma$ is \emph{$n$-weak cop win} if they can choose $\sigma$ and $\rho$ such that $\Gamma$ is always CopWin$(n,\sigma,\rho,\psi,R)$ for any $\psi$ and $R$ chosen by the robber. Symbolically,
    \[ \Gamma \text{ is } n\text{-weak cop win} \iff \exists\; \sigma, \rho\; \forall\; \psi, R: \Gamma \text{ is CopWin}(n,\sigma,\rho,\psi,R). \]
    The \emph{weak cop number} of $\Gamma$, denoted $\wCop(\Gamma)$, is the smallest $n$ such that $\Gamma$ is $n$-weak cop win. If no such integer exists, we say $\wCop(\Gamma) = \infty$.

\section{The weak cop number of wreath products}
\label{sec:Lamp}

\subsection{The Lamplighter game}  

Lamplighter is a perfect information game played with $P \ge 2$ players on a single underlying object called a \emph{streetmap}, essentially a graph where the vertices represent lamps that can take on a variety of states. One player, the \emph{lamplighter}, moves along the graph, changing the states of lamps. The other $P-1$ players, the \emph{copiers}, move along their own graphs, working together in an attempt for one of them to approximate the lamplighter's pattern of lamps. Before describing the game, we introduce some terminology.

\begin{definition}[Streetmap]
    \label{def:omega-streetmap}
    A \emph{streetmap} is a triple $M = (\Omega, \omega, \Lambda)$, where $\Omega$ and $\Lambda$ are simple, connected graphs, and $\omega$ is some distinguished point of $\Omega$. The vertices of $\Lambda$ are called \emph{lamps}, and the vertices of $\Omega$ are called \emph{states}.
\end{definition}

    The reader is encouraged to think about the lamp states as different ways the lamps can be ``lit". For example, when $\Omega$ is the $2$-path with vertices labelled 0 and 1, one can imagine that a lamp being in state 0 means it is unlit, and that a lamp in state 1 is lit. When $\Omega$ has more than two vertices, one can consider the nonzero states as being different \emph{colours} that the lamp can take, and the adjacency relation on $\Omega$ describes how the colors of the lamp can be changed. The vertex $\omega$ of $\Omega$ can be thought of as the default state of a lamp. 
    
\begin{definition}[Board]    
    Let $M$ be a streetmap. An \emph{M-Board} is a triple $B = (M, p, \phi)$, where $p \in V(\Lambda)$ and $\phi:V(\Lambda) \rightarrow V(\Omega)$ is a mapping with $\phi(v) = \omega$ for all but finitely many $v \in V(\Lambda)$.    The lamp $p$ is called the player's \emph{position} (in the street map), and for each $v \in V(\Lambda)$, $\phi(v)$ is called the \emph{state} of the lamp $v$.
\end{definition}

At any stage of the game, each player is assigned a board. The board of each player changes with the moves that the player does. 

\begin{figure}[t] 
    \centering    \includegraphics[width=\textwidth]{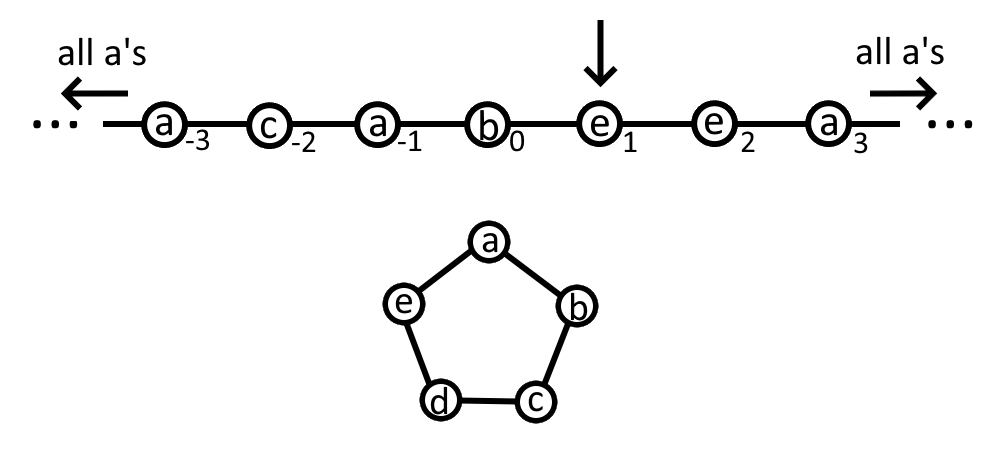}
    \caption{An illustration of an $M$-board over the streetmap $M=(C_5,a, P_\infty)$.}   \label{fig:lamplighter-as-omega}
\end{figure}

\begin{example}
Consider the $5$-cycle $C_5$ and label its  vertices by $a,b,c,d,e$,  the infinite path $P_\infty$ with vertices labelled by $\mathbb{Z}$ in the natural way, and let
\[\phi\colon \mathbb{Z} \to \{a,b,c,d,e\}, \qquad \phi(n)= \left\{
  \begin{array}{ll} 
      a & \text{if }|n|\geq 3 \text{ or }   n=-1,\\
      b & \text{if } n=0,\\
      c & \text{if } n=-2,\\
      e & \text{if } n=1,2.
      \end{array}
\right.\]
Figure~\ref{fig:lamplighter-as-omega} illustrates the street map $M=(C_5,a, P_\infty)$ and the $M$-Board $(M,1,\phi)$. In the illustration, each lamp is represented by a circle; 
the vertex labels of the lamps are below and to the right of the corresponding circle; the state of the lamp $\phi(v)$ is inside the corresponding circle. The downwards arrow represents the player's position.
\end{example}

Now we are ready to describe the game. 

\subsubsection{Initial setup of the Lamplighter game.}

Lamplighter is a perfect information game played with $P \ge 2$ players on   a \emph{streetmap} $M$. There is a distinguished player called the \emph{Lamplighter} and the other $P-1$ players are called \emph{copiers}.

    \begin{itemize}
        \item The copiers collectively choose two positive integers $\rho$ and $\sigma$, called the \emph{copier reach} and \emph{copier speed}, respectively. The speed is a measure of how many moves each copier can take in a single \emph{turn}.
        \item Knowing the values of $\rho$ and $\sigma$, the lamplighter chooses positive integers $\psi$ and $r$ called the \emph{lamplighter speed} and \emph{radius of play}, respectively, as well as a vertex $v$ of $\Lambda$ called the  \emph{center of the area of play}. The \emph{area of play} is defined as the ball of radius $r$ centered at $v$ in $\Lambda$ which we denote by $\Lambda_r(v)$.

        \item The copiers each choose a starting board for themselves, i.e., they select their initial position $p$ and the initial states of their lamps $\phi$.
        \item The lamplighter, knowing the starting board for each of the copiers, chooses his starting board. All of the lamps that are not in state $\omega$ on this board, as well as the lamplighter's position, must lie within the area of play.
    \end{itemize}

  \begin{figure}[t] \includegraphics[width=0.5\textwidth]{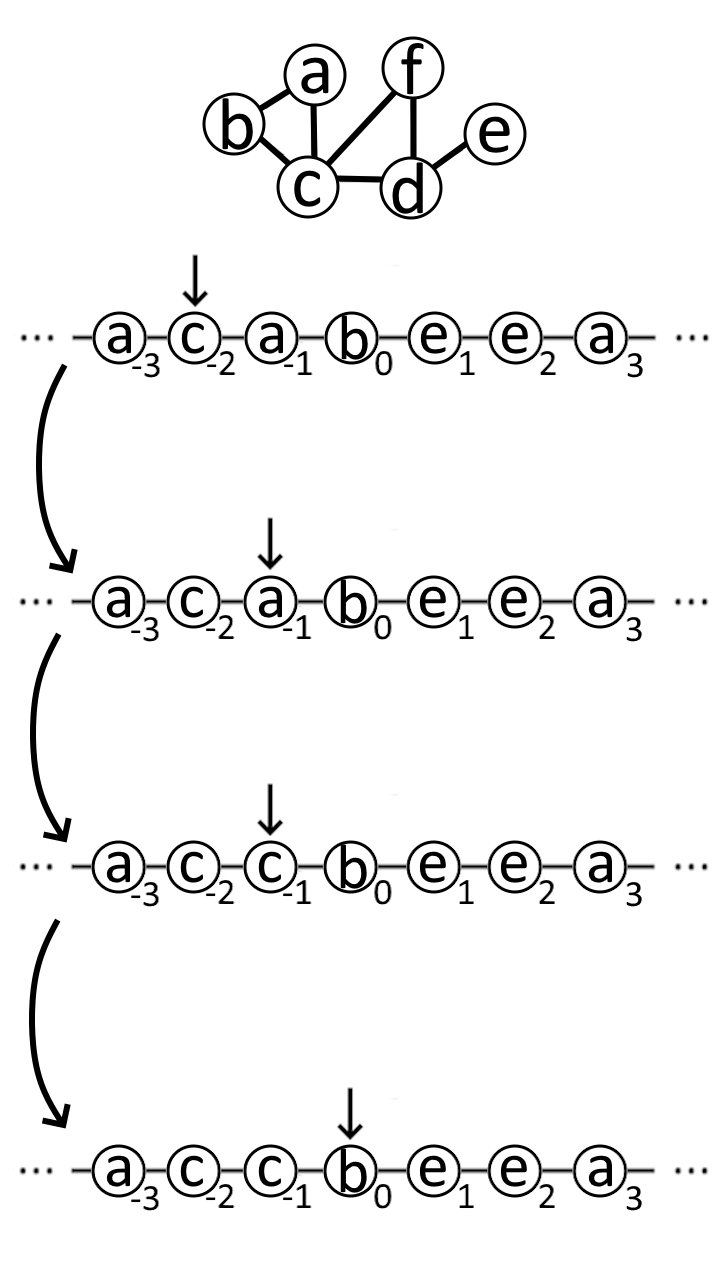}
        \caption{An example of a turn in Lamplighter. In this case there is a move of type 1, follow by a move of type 2, and then a move of type 1 again. }
        \label{fig:lamplighter-turn}
    \end{figure}

\subsubsection{Turns and player movement}
    After all of the parameters are selected, play happens in turns, starting with the copiers. On the copiers' turn, they all play in unison. The copiers are able to communicate with one another and coordinate their moves.
    The copiers, on each of their turns, move up to a total of $\sigma$ times (being able to select any of their available moves each time), followed by the lamplighter similarly moving up to $\psi$ times.  
      
  
    Let $B = (M,p,\phi)$ be the player's board on a given turn. There are two types of moves, each one changing the player's board:
   
    \begin{itemize}
        \item     The first type of move is to change positions to a new lamp $p'$ that is adjacent to $p$, in other words, the new board becomes 
    \[ B' = (M,p',\phi), \text{ where } p' \in \Lambda_r(v) \text{ such that } \{p,p'\} \in E(\Lambda). \]
        If the lamplighter chooses this option, the lamp $p'$ must be in the area of play.

        \item The second possible type of move is to change  the state of the lamp at their current position in accordance to the relation given by the edges of $\Omega$. That is, if the state of the lamp at the player's current position is $s=\phi(p)$, he can choose to change the lamp to a new state $s'$, where
        $s$ and $s'$ are adjacent vertices in $\Omega$. In this case the new board becomes       
    \[ B' = (M,p,\phi'), \text{ where } \phi'(v) = \begin{cases}
        s' &\text{if } v = p, \\
        \phi(v) &\text{otherwise.}
    \end{cases}\]
    \end{itemize}

    \begin{example}\label{ex:Second}
    Let $\Omega$ be the graph with six vertices labelled by $a,b,c,d,e,f$ illustrated in Figure~\ref{fig:lamplighter-turn}. Consider the streetmap $M=(\Omega, a, P_\infty)$. The figure illustrates the turn of a player in the Lamplighter game, where the given player has speed $\ge 3$ by showing  player's board at the beginning of the turn plus the boards after three moves.  
    \end{example}

\subsubsection{Lamplighter move restriction}

The lamplighter  plays the entire game within the area of play. That means at each stage of the game, if $(M,p,\phi)$ is a board representing the lamplighter then $p$ is a vertex in the area of play.  

\subsubsection{Win conditions}

    Given a streetmap $M$, we denote by $\mathcal{B}(M)$ the set of all possible boards on $M$. Define the \emph{distance} between two boards in $\mathcal{B}(M)$ to be the minimum number of moves required to change one board into the other. Note that this defines a metric on $\mathcal{B}(M)$.

    The goal of the copiers is to come ``close enough" to copying the lamplighter's board. Specifically, the \emph{copiers win} if at any point during either player's turn, the lamplighter's board is at distance $\rho$ or less from any of the copiers' boards. Conversely, the lamplighter wins if he can devise a strategy that will let him avoid coming within the copiers' reach indefinitely.

    \begin{example} Consider the Lamplighter game on the street map $M$ defined in Example~\ref{ex:Second} with three copiers.  Figure~\ref{fig:copier-win-example} shows the boards of all the players at some stage of the game. If $\rho \ge 3$ in this game, then the copiers have won here, since the third copier $c_3$ is three moves away from copying the lamplighter's board (change lamp $-1$ to state $a$, move to lamp $-2$, change lamp $-2$ to state $c$).
    \end{example}

    \begin{figure}
        \centering
        \includegraphics[width=\textwidth]{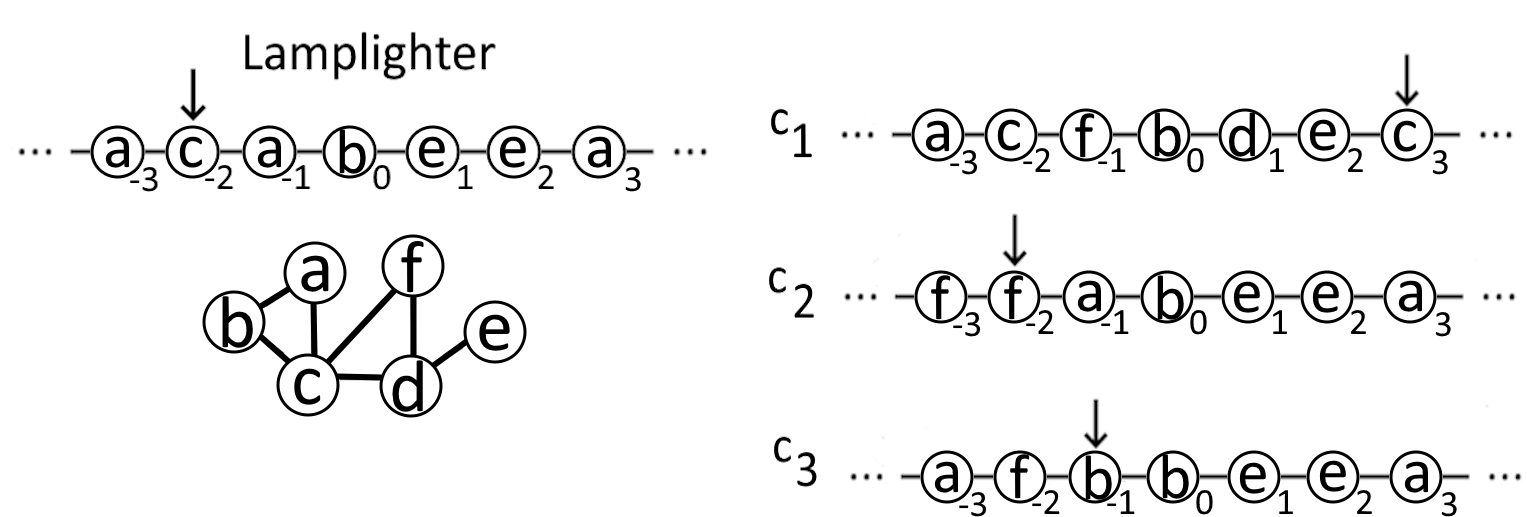}
        \caption{An example of a game state in Lamplighter. If in this game $\rho \ge 3$, then the copiers have won.}
        \label{fig:copier-win-example}
    \end{figure}

\subsubsection{The copier-win number $\wCop^*(M)$}
 If a streetmap $M$ admits a strategy that allows $n$ copiers to choose values of $\rho$ and $\sigma$ that will always allow them to win, no matter what values of $\psi$, $r$, and $v$ the lamplighter chooses, then the streetmap $M$ is called \emph{$n$-copier win}. The \emph{copier-win number} of $M$, which we will denote $\wCop^*(M)$, is the smallest positive integer $n$ for which $M$ is $n$-copier win. If $M$ is not $n$-copier win for any $n \ge 1$, we say $\wCop^*(M) = \infty$.

\subsection{General winning strategy for the Lamplighter}

    \begin{theorem} \label{thm:omega-lamplighter-strategy}
        Let $M = (\Omega, \omega, \Lambda)$ be a streetmap. If $\Lambda$ is a connected graph with infinite diameter and $\Omega$ is  nontrivial connected graph, then for any $n \ge 1$, the lamplighter has a winning strategy in a game of Lamplighter on $M$ against $n$ cops, i.e. $\wCop^*(M) = \infty$.
    \end{theorem}

    \begin{proof}
\textbf{Choosing the lamplighter's parameters:} Suppose, at the beginning of the game, the copiers choose their reach and speed to be $\rho$ and $\sigma$, respectively. Then the lamplighter chooses speed $\psi = 3n + \sigma + \rho + 1$ and radius of play $r = \left\lceil\frac{\sigma+\rho}{2}\right\rceil + n$. 

Since $\Lambda$ has infinite diameter, it contains a geodesic path $P$ consisting of $\sigma + \rho + 2n$ lamps.
In particular, for any pair of vertices in $P$, the distance in $\Lambda$ equals their distance in $P$; so there are no ``shortcuts" in $\Lambda$ between them. Define the center of the area of play $v$ the central vertex  of the path $P$, and observe that the vertices of $P$ are contained in the area of play. 
From one endpoint of $P$ to the other, in order, label these lamps
        \[ \ell_1, \ell_2, \ldots, \ell_n, m_1, m_2, \ldots, m_{\sigma+\rho}, r_1, r_2, \ldots, r_n. \]

Since $\Omega$ is nontrivial and connected, we can pick a state $ \omega_1 \in V(\Omega)$ such that $\omega$ and $\omega_1$ are distinct adjacent  vertices in $\Omega$.
From here on, we denote the states $\omega$ and $\omega_1$ as $0$ and $1$ respectively.

After the copiers choose their initial boards, the lamplighter chooses his initial board as follows. Denote the copiers by $c_1, c_2, \ldots, c_n$ in some arbitrary order, and let $B_i = (M,p_i,\phi_i)$ be the initial board of copier $c_i$ for $i = 1, \ldots, n$. Then the lamplighter chooses his initial board $B_0$ to be
        \[ B_0 = (M,\ell_1,\phi_0), \text{ with } \phi_0(v) = \begin{cases}
            1 &\text{if } v \in \{\ell_i, r_i\} \text{ for some } i \le n \text{ and } \phi_i(v) = 0, \\
            0 &\text{otherwise.}
        \end{cases} \]
        Note that for each $j = 1, \ldots, n$, we have that $\phi_0(\ell_j) \neq \phi_j(\ell_j)$ and 
        $\phi_0(r_j) \neq \phi_j(r_j)$. Moreover, the lamplighter starts the game at the end of the path $P$ labelled by $\ell_1$.
        
        \textbf{The lamplighter's strategy:} On each of his turns, he walks  to the opposite end of the path (i.e., from $\ell_1$ to $r_n$, or from $r_n$ to $\ell_1$). As he travels, if he encounters a vertex $\ell_i$ or $r_i$ in $\{\ell_1, \ell_2, \ldots, \ell_n, r_1, r_2, \ldots, r_n\}$ that is the same state as the corresponding vertex in $B_i$, switch it to the opposite state (0 or 1). This guarantees that, at the beginning of the copiers' turn, for each copier $c_i$ there are at least two lamps ($\ell_i$ and $r_i$) whose states are different from the corresponding lamp on the lamplighter's board. Observe that the lamplighter will be playing the entire game within the area of play. 

        \textbf{Proof that the lamplighter's speed $\psi$ is sufficient:} In moving from one end of the path of lamps to the other, he must move between $2r+1$ lamps, and therefore requires $2r = 2\left(\left\lceil\frac{\sigma+\rho}{2}\right\rceil + n\right) \le \sigma + \rho + 1 + 2n$ changes in position. Additionally, he has to change the states of at most $n$ lamps, since at most one lamp per copier will differ in state from that copier's board. Therefore, we can complete the given strategy in at most $3n + \sigma + \rho + 1$ moves per turn, and our chosen speed is sufficient.
        
        \textbf{Proof that the lamplighter wins with the above strategy:} We want to show that if, on the copiers' turn, no copier $c_i$ can change their board to come within distance $\rho$ of the lamplighter's board. Namely, if at the beginning of the copiers' turn, $\phi_i(\ell_i) \neq \phi_0(\ell_i)$ and $\phi_i(r_i) \neq \phi_0(r_i)$, then during the entirety of the copiers' turn, the distance between the lamplighter's board and any of the copier's boards is strictly greater than $\rho$.

        As the copiers begin their turn, choose an arbitrary copier $c_i$. In order to match the lamplighter's board, $c_i$ must switch both the lights $\ell_i$ and $r_i$. However, no matter what position $c_i$ is in at the beginning of the turn, they need to use at least 1 move to change the state of the lamp, and then traverse through the middle sequence of $\rho + \sigma$ lamps, which requires at least $\rho + \sigma + 1$ moves. However, they only \emph{have} at most $\sigma$ moves available to them, which means at the end of their turn they still require at least $(\rho+\sigma+2)-\sigma = \rho + 2 > \rho$ moves to match the lamplighter's board. Therefore every copier ends their turn at distance strictly greater than $\rho$ from the lamplighter, and thus the copiers do not win on their turn.

        We must additionally show that the copiers do not win during the \emph{lamplighter's} turn. In other words, we must show that the lamplighter can execute the above strategy without coming within distance $\rho$ of a copier.

        Note that, for each copier $c_i$, the argument above only relies on the position of $c_i$ and the states of lights $\ell_i$ and $r_i$. In particular, the position of the lamplighter does not affect this fact, nor do the states of the other lamps. Therefore, changing the lamplighter's position does not cause the copiers to win. Additionally, this means that changing the state of a lamp $\ell_i$ or $r_i$ does not bring any copier $c_j$ with $j \neq i$ within distance $\rho$ of the lamplighter; furthermore, it strictly \emph{increases} the distance from the lamplighter to copier $c_i$. Thus, the copiers do not win during the lamplighter's turn.
        
        Therefore, the lamplighter can avoid the copiers indefinitely for any number $n$ of copiers, and so $\wCop^*(M) = \infty$.
    \end{proof}

 Let us record a few properties of the strategy that we described in the proof above, since they are used in the proof of the main result of the section. 
 \begin{remark}\label{rem:Proof}
In the proof of Theorem~\ref{thm:omega-lamplighter-strategy}, the strategy for the lamplighter 
playing on $M$ against $n$ copiers with speed $\sigma$ and reach $\rho$ has the following properties: 
\begin{itemize}           \item The lamplighter chooses his speed to be $\psi = 3n + \sigma + \rho + 1$,  and his radius of play to be $r = \left\lceil\frac{\sigma+\rho}{2}\right\rceil + n$. 
\item The lamplighter fixes a geodesic path $P$ in $\Lambda$ with $2n+\sigma+\rho$ vertices, and lets the center of the area of play be the central vertex of $P$.  

\item The lamplighter fixes a state $\omega_1$ such that $\omega$ and $\omega_1$ are distinct, adjacent vertices in $\Omega$.

\item Any board $(M,p,\phi)$ representing the lamplighter during any stage of the game satisfies that: $p$ is a vertex of $P$, $\phi$ only takes values in $\{\omega,\omega_1\}$, and $\phi^{-1}(\omega_1)$ is a finite subset of vertices of $P$. 

\item If $(M,v,0)$ is a board where $v$ is the central  vertex of $P$ and $0$ represents the constant function $V(\Lambda)\to \{\omega\}$, then distance between  $(M,v,0)$ and any board $(M,p,\phi)$ representing the lamplighter at some stage of the game is at most $6r+1$.

Indeed to move from $(M,v,0)$ to $(M,p,\phi)$ one needs to first use $r$ moves to move to an endpoint of $P$ at maximal distance from $p$, then transverse the path $P$ to the other end while correcting the discrepancies between $0$ and $\phi$ which require at most $4r+1$ moves, and finally move back to $p$ which needs at most $r$ moves.

\item Analogously, if $(M,p_1,\phi_1)$ and $(M,p_2,\phi_2)$ are boards representing  the lamplighter during distinct stages in the game then 
$(M,p_1,\phi_1)$ and $(M,p_2,\phi_2)$ are at distance at most $R=2(6r+1)$. 
\end{itemize}        
 \end{remark}




\subsection{The wreath product of graphs}

In this section, we defined the \emph{(restricted) wreath product} of graphs, a notion that  is heavily related to the Lamplighter game.   We first introduce some notation.

\begin{definition}\label{def:NewNotation}
    Let $X$ and $Y$ be sets, where $X$ has some distinguished base element $a$. We define
    \[ X_{a}^{(Y)} = \{ f\colon Y \rightarrow X \mid f(y) = a \text{ for all but finitely many values of } y \}. \]
    In other words, we can view $X_{a}^{(Y)}$ as either the set of finitely-supported \emph{functions} from $Y$ to $X$ based at $a$, or as the set of finitely-supported \emph{$X$-sequences} indexed by $Y$, based at $a$.
\end{definition}

Note that we often use two interchangeably notations for  a function $f\colon Y \rightarrow X$, namely, as the sequence $(f(y))_{y \in Y}$ or as a sequence $(x_y)_{y \in Y}$ with $f(y) = x_y$. The following definition is a version of the one given by Donno \cite{Donno2013}, expanded to include infinite graphs.

\begin{definition}[Restricted Wreath Product of Graphs]
    \label{def:graph-wreath-product}
    Let $\Lambda = (V,E)$ and $\Omega = (W,F)$ be graphs, where $\Omega$ has base point $\omega$. The \emph{(restricted) wreath product} is the graph $\Omega \wr \Lambda$ whose vertex set is the Cartesian product $W_{\omega}^{(V)} \times V$, where two vertices $(f,v)$ and $(f',v') \in V(\Omega \wr \Lambda)$ are adjacent if:
    \begin{enumerate}
        \item $v = v' =: \overline{v}$, $f(w) = f'(w)$ for every $w \neq \overline{v}$, and $\{f(\overline{v}),f'(\overline{v})\} \in F$. An edge of this type is called an edge of type 1.
        \item $f = f'$ and $\{v,v'\} \in E$ . An edge of this type is called an edge of type 2.
    \end{enumerate}
\end{definition}

\begin{figure}[!t]
\centering
\begin{subfigure}[b]{\textwidth}
    \centering
    \raisebox{0cm}{\includegraphics[width=\textwidth]{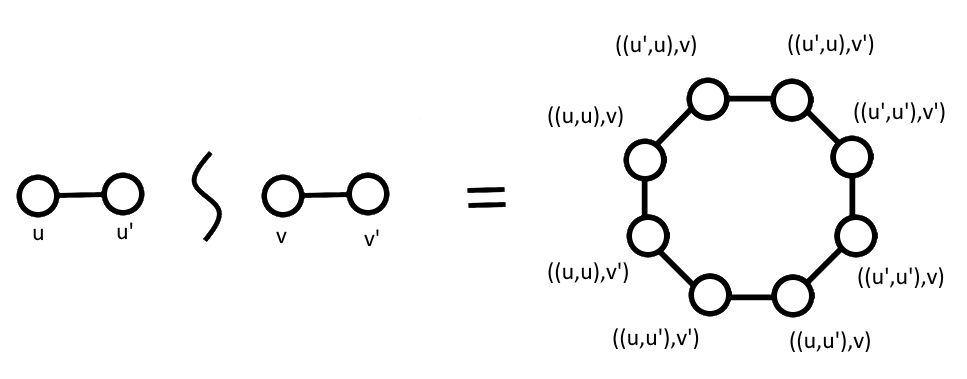}}
    \caption{The graph wreath product $P_2 \wr P_2$ is isomorphic to $C_8$.}
    \label{fig:wreath-product-ex-1}
\end{subfigure}
\hspace{1cm}
\begin{subfigure}[b]{0.9\textwidth}
    \centering
    \includegraphics[width=\textwidth]{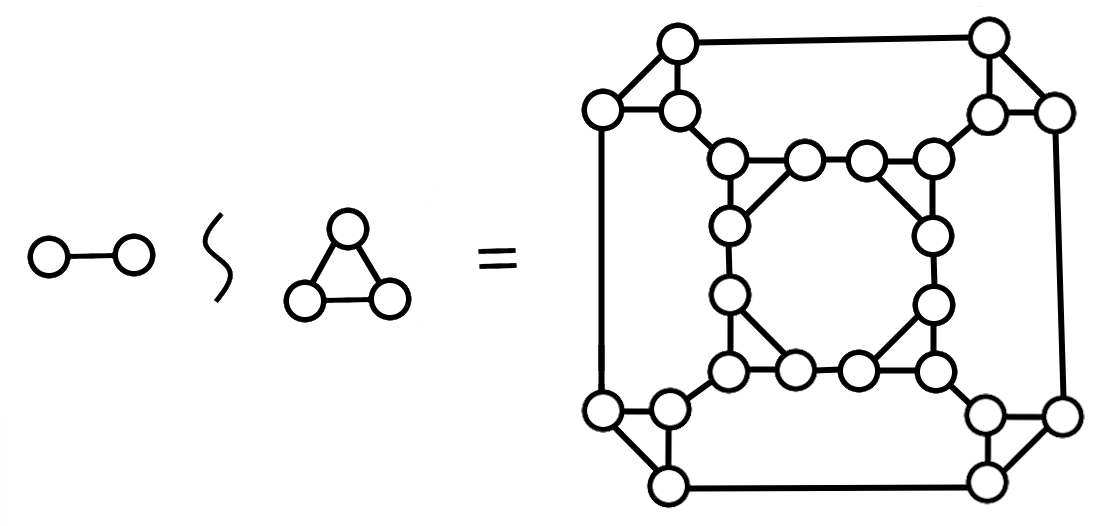}
    \caption{The graph wreath product $P_2 \wr C_3$ is isomorphic to the truncated cube graph.}
    \label{fig:wreath-product-ex-2}
\end{subfigure}
\caption{Two examples of the wreath products of small graphs.}    \label{fig:wreath-product002}
\end{figure}

A pair of examples illustrating the wreath product of some finite graphs is showed in Figure~\ref{fig:wreath-product002}. In our context, the wreath product of infinite graphs is more relevant but their geometry is complex and in general difficult to visualize.

\subsection{Relation Between Lamplighter and Weak cops and Robbers}

In order to use our Lamplighter game to derive results about Weak cops and robbers, we must establish some additional theory.

\begin{definition}
    Let $M = (\Omega, \omega, \Lambda)$ be a streetmap. Define a graph $\Gamma_{\mathcal{B}}(M)$ as follows.
    \begin{enumerate}
        \item $V(\Gamma_{\mathcal{B}}(M)) = \mathcal{B}(M)$, the set of possible $M$-boards, and
        \item two boards are adjacent in $\Gamma_{\mathcal{B}}(M)$ if they differ by a single move.
    \end{enumerate}
    We call $\Gamma_{\mathcal{B}}(M)$ the \emph{graph of $M$-boards}.
\end{definition}

\begin{prop}
    \label{prop:omega-graph-compare-1}
    Let $M = (\Omega, \omega, \Lambda)$ be a streetmap. Then $\Gamma_{\mathcal{B}}(M) \cong (\Omega \wr \Lambda)$.
\end{prop}

\begin{proof}
    Let $\Omega = (V,E)$ and $\Lambda = (W,F)$. Define the mapping
    \begin{align*}
        \Phi: \mathcal{B}(M) &\rightarrow V(\Omega \wr \Lambda) \\
        (M, p, \phi) &\mapsto (\phi, p),
    \end{align*}
    which is well-defined since $\phi \in W^{(V)}_{\omega}$ and $p \in V$. This is a bijection since it has an obvious inverse
    \begin{align*}
        \Phi^{-1}: V(\Omega \wr \Lambda) &\rightarrow \mathcal{B}(M) \\
        (f, v) &\mapsto (M, v, f).
    \end{align*}
    Now all that is left to prove is that the edges induced by $\Phi$ are the same as those in the graph $\Omega \wr \Lambda$. Let $(M,p,\phi) \in \mathcal{B}(M)$. Then there are two types of edges from $(M,p,\phi)$.
    \begin{enumerate} 
        \item First, there are the edges corresponding to changing the state of the lamp at the current position, to an adjacent state in $\Omega$. This means that the given edge is connected to $(M,p,\phi')$, where $\{\phi'(p), \phi(p)\} \in F$ and $\phi'(v) = \phi(v)$ for $v \neq p$. This description makes it clear that the target edge $\{(\phi,p), (\phi',p)\}$ is exactly an edge of type 1 in $\Omega \wr \Lambda$.
        
        \item Secondly, there are the edges corresponding to moves that change positions. For a position $p \in V$, the player can move to a position $p' \in V$ with $\{p,p'\} \in E$. Therefore such an edge is connected to $(M,p',\phi)$, and therefore this edge is mapped to $\{(\phi,p), (\phi,p')\}$ in $\Omega \wr \Lambda$ with $\{p,p'\} \in E$, which are exactly the edges of type 2 in $\Omega \wr \Lambda$.
    \end{enumerate}
    Therefore $\Phi$ is a graph isomorphism from $\Gamma_{\mathcal{B}}(M)$ to $\Omega \wr \Lambda$.
\end{proof}

 There is a straightforward connection between the Lamplighter game and the Weak cops and robbers game, described as follows.
 
\begin{proposition}
\label{thm:strategy-transfer2}
    Let $\Lambda$ be a  connected graph with infinite diameter and $\Omega$ a nontrivial connected graph. If $M = (\Omega, \omega, \Lambda)$ is a streetmap and $n$ a positive integer, then  
    the winning strategy for the lamplighter given by Theorem~\ref{thm:omega-lamplighter-strategy} for the  game of Lamplighter on $M$ with $n$ copiers, provides a winning strategy for the robber in a game of Weak Cops and Robbers on $\Omega \wr \Lambda$ with $n$ cops.
\end{proposition}

\begin{proof}
     By the isomorphism in Proposition \ref{prop:omega-graph-compare-1}, we regard the vertices of $\Omega\wr\Lambda$ as $M$-boards and adjacency defined according to the moves in the Lamplighter game. 
     
     It follows that moving in a game of Lamplighter on $M$ (either changing positions or changing states) is equivalent to moving between vertices in Weak cops and robbers on $\Omega \wr \Lambda$.     
    Consequently, the \emph{lamplighter speed}, \emph{copier reach}, and \emph{copier speed} in the Lamplighter game  correspond exactly with the \emph{robber speed},  \emph{reach}, and \emph{cop speed}, respectively in the Weak Cops and Robbers game.
    The only parameters that do not align exactly are the \emph{radius of the area of play} and the \emph{center of the area of play}. 
    According to the winning strategy for the lamplighter given by Theorem~\ref{thm:omega-lamplighter-strategy} via Remark~\ref{rem:Proof}, we have that $r = \left\lceil\frac{\sigma+\rho}{2}\right\rceil + n$ is  the radius of the area of play,  and the  center of the area of play $v\in V(\Lambda)$  is at the center of a geodesic path $P$ in $\Lambda$ with $2r+1$ vertices. Moreover,
    if $(M,p_1,\phi_1)$ and $(M,p_2,\phi_2)$ are boards representing  the lamplighter during distinct stages in the game then the distance between them is bounded from above by $R=2(6r+1)$. 

The center of the area of play in the Weak Cops and Robbers game is defined to be the board $(M,v,0)$ where $0$ denotes the constant function $V(\Lambda) \to \{w\}$ and the radius of the area of play is defined as $R$. Note that Remark~\ref{rem:Proof} states that any board representing the lamplighter is at distance at most $R$ from $(M,v,0)$.  

Since the vertices of $\Omega\wr\Lambda$ have been identified with boards, the strategy for the robber in the Weak Cops and robbers game is defined as our winning strategy for the lamplighter against $n$ copiers that reproduce the moves of the cops in the Weak Cops and Robbers game.  Then, since at any stage of the Lamplighter game, the board representing the lamplighter is at distance larger than $\rho$ than any board representing a copier, we have that the robber is never captured. We show in the previous paragraph that any stage of the game, the board representing the  lamplighter (and hence the robber) is in the $R$-ball about the center of play $(M,v,0)$ in $\Omega\wr\Lambda$ the Weak Cops and Robbers game. Thus, this constitutes a winning strategy for the robber.
\end{proof}

    The above results gives us the following   corollary that shows that the weak cop number of a larger class of wreath products of graphs is infinite.

\begin{corollary}
    \label{cor:cop-number-compare}
  If $\Lambda$ be a connected graph with infinite diameter  and $\Omega$ is a nontrivial connected graph  then $\wCop(\mathit{\Omega \wr \Lambda})=\infty$.
\end{corollary} 

\subsection{Weak cop number of wreath-products of groups}

Let us recall the definition of the restricted wreath product of groups. 
Given groups $G$ and $H$, 
the \emph{restricted wreath product} of  $G$ by $H$, denoted by $G \wr H$, is defined as the group on the set $\bigoplus\limits_{H} G  \rtimes H$ with operation
    \[ \left((g_h)_{h \in H}, h_1\right)\left((g'_h)_{h \in H}, h_2\right) = \left((g_{h}g'_{h_1^{-1}h})_{h \in H}, h_1h_2\right). \]
The groups $G$ and $H$ have natural identifications as  subgroups of $G \wr H$ given by 
  \begin{align*}
        \iota_G\colon  G &\hookrightarrow G \wr H \\
        g &\mapsto ((g^{\delta(e_H,h)})_{h \in H}, e_H) \\
        \\
        \iota_H\colon H &\hookrightarrow G \wr H \\
        h &\mapsto ((e_G)_{h \in H}, h)
    \end{align*}
where $\delta(e_H,h) = 1$ if $h = e_H$, the identity of $H$, and is $0$ otherwise. It is an exercise to verify that if $S$ and $T$ generating sets of $G$ and $H$ then $ \iota_G(S) \cup  \iota_H(T)$, which we will denote $S \cup T$, is a finite generating set for $G \wr H$. In particular, if $G$ and $H$ are finitely generated, then $G\wr H$ is finitely generated.    A notable example  is $\mathbb{Z}_2 \wr \mathbb{Z}$ which is known as the \emph{Lamplighter group}.

\begin{prop}
    \label{prop:wreath-product-link}
    Let $G$ and $H$ be groups with finite generating sets $S$ and $T$, respectively. Then $\Cay(G,S) \wr \Cay(H,T) = \Cay\left(G \wr H, S \cup T\right)$.
\end{prop}

\begin{proof}
    Note that $V(\Cay(G,S)) = G$ and $V(\Cay(H,T)) = H$, so $V(\Cay(G,S) \wr \Cay(H,T)) = G^{(H)} \times H = V(\Cay(G \wr H, S \cup T))$. Now we need only show that the edge sets of the two graphs are equal. Take an arbitrary element $x = ((g_h)_{h \in H}, h_0) \in G^{(H)} \times H$.
    \begin{enumerate}
        \item There are the edges of type 1 in $\Cay(G,S) \wr \Cay(H,T)$. For each $s \in S$, the vertex $(\phi,h_0)$ is adjacent to the vertex $(\phi',h_0)$, where $\phi'(h_0) = \phi(h_0)s$, and $\phi'(h) = \phi(h)$ for all $h \in H \setminus \{h_0\}$. This aligns with edges of the form $x\iota_G(s)$ in $\Cay(G \wr H, S \cup T)$, since
        \begin{align*}
            x\iota_G(s) &= ((\phi(h))_{h \in H},h_0)((s^{\delta(e_H,h)})_{h \in H},e_H) \\
            &= ((\phi(h)s^{\delta(e_H,h_0^{-1}h)})_{h \in H}, h_0) \\
            &= (\phi',h_0).
        \end{align*}
        \item Secondly, we have the edges of type 2. For each $t \in T$, the vertex $(\phi,h_0)$ is adjacent to the vertex $(\phi,h_0t)$. This aligns with edges of the form $x\iota_H(t)$ in $\Cay(G \wr H, S \cup T)$, since
        \begin{align*}
            x\iota_H(t) &= ((\phi(h))_{h \in H},h_0)((e_H)_{h \in H},t) \\
            &= (\phi(h)_{h \in H},h_0t) \\
            &= (\phi,h_0t).
        \end{align*}
    \end{enumerate}
    
       Hence $\Cay(G,S) \wr \Cay(H,T)$ and $\Cay(G \wr H, S \cup T)$ are  isomorphic as graphs.
\end{proof}

We now have all of the tools we need to derive our main result.

\begin{theorem}
    \label{thm:wreath-prod-wcop}
    Let $G$ and $H$ be finitely generated groups. If $G$ is nontrivial and $H$ is infinite, then $\wCop(G \wr H) = \infty$.
\end{theorem}

\begin{proof}
    Pick finite generating sets $S$ and $T$ for $G$ and $H$, respectively. By Corollary~\ref{cor:cop-number-compare}, $\wCop((\Cay(G,S) \wr \Cay(H,T)) = \infty$. Since $\Cay(G,S) \wr \Cay(H,T) = \Cay(G \wr S, S \cup T)$ by Proposition \ref{prop:wreath-product-link}, we have $\wCop(\Cay(G \wr S, S \cup T)) = \infty$, and therefore $\wCop(G \wr H) = \infty$.
\end{proof}

\section{Thompson's group $F$ has infinite weak cop number}\label{sec:Francesco}

In this part we describe an argument communicated to the authors by Francesco Fournier-Facio.

One can deduce that $\wCop(F)=\infty$, from the following three statements.

\begin{theorem}\cite[Theorem E]{Lee2023}
Let $H$ be a subgroup of a be a finitely generated group $G$. If $H$ is a retract of $G$, then $\wCop(H)\leq \wCop(G)$.
\end{theorem}

\begin{theorem}\cite[Theorem C]{Lee2023}
$\wCop(\mathbb{Z}^2)=\infty$.    
\end{theorem}

The following statement is well-known among experts on Thompson's group $F$, and is implicit in expositions on generalizations of Thompson's groups, see~\cite{bieri2016groups} or the recent work~\cite[Cor. 3.10]{balasubramanya2024hyperbolic}. 

\begin{theorem}
Thompson's group $F$ retracts onto $\mathbb{Z}^2$.  
\end{theorem}

Let us share an explanation of the above statement based on~\cite[Sec. 3]{balasubramanya2024hyperbolic}. Regard Thompson’s group $F$   as the
subgroup of homemorphisms of the unit interval consisting of orientation-preserving, piecewise linear homeomorphisms whose non-differentiable points are dyadic rationals and whose slopes are all powers of $2$. Given an element $g\in F$, let $\chi_0(g) = \log_2 g'(0)$ and $\chi_1(g)=\log_2 g'(1)$, where $g'(0)$ and $g'(1)$ denote the slopes of the piece-wise linear homemorphism $g$ at $0$ and at $1$ respectively. Now we show that the group homomorphism   $r\colon F\to \mathbb{Z}^2$ given by $r=(\chi_1,\chi_2)$ is a retraction.  Let $g_0\in F$ such that $g_0'(0)=2$ and   $\mathsf{Supp}(g_0) = \{x \in [0, 1] \mid g_0(x) \neq x\} = [0,1/2]$; then, let $g_1\in F$ such that $g_1'(1)=2$ and $\mathsf{Supp}(g_2)=[1/2,1]$ . Observe that the subgroup $H=\langle g_0,g_1\rangle$ is isomorphic to $\mathbb{Z}^2$, $r(g_0)=(1,0)$, and   $r(g_1)=(0,1)$. Hence $r$ is a retraction.

The Bieri-Strebel groups are generalizations of Thompson's group, see~\cite{bieri2016groups} for definitions. An analogous argument provides a retraction $G\to \mathbb{Z}^2$ for $G$ any  Bieri-Strebel group with a finitely generated slope group. Hence any group in this class has infinite weak cop number.

\bibliographystyle{alpha} \bibliography{xbib}

\end{document}